\documentclass[reqno]{amsart}
\usepackage{hyperref}
\usepackage{amssymb}

\begin{document}
\title[On a degenerate singular elliptic problem]{\hfilneg \hfilneg 
{On a degenerate singular elliptic problem}}

\author[P. Garain] 
{Prashanta Garain}

\address{Prashanta Garain \newline
Department of Mathematics and Systems Analysis,
Aalto University, Otakaari 1, 02150\\
Espoo, Finland}
\email{pgarain92@gmail.com}

\subjclass[2010]{35J70, 35J75, 35D30}
\keywords{Degenerate Elliptic Equation; Singular Nonlinearity; Muckenhoupt Weight; Weighted Sobolev Space 
\hfill\break\indent}

\begin{abstract}
In this article we provide existence, uniqueness and regularity results of a degenerate singular elliptic boundary value problem whose prototype is given by
\begin{gather*}
               \begin{cases}               
              -\operatorname{div}(w(x)|\nabla u|^{p-2}\nabla u)=\frac{f(x)}{u^\delta}\,\,\text{ in }\,\,\Omega,\\
             u>0\text{ in }\Omega,\\ u = 0  \text{ on } \partial\Omega,
             \end{cases}
\end{gather*}
where $\Omega$ is a bounded smooth domain in $\mathbb{R}^N$ with $N\geq 2$, $w$ belong to the Muckenhoupt class $A_p$ for some $1<p<\infty$, $f$ is a nonnegative function belong to some Lebesgue space and $\delta>0$.
\end{abstract}

\maketitle
\numberwithin{equation}{section}
\newtheorem{theorem}{Theorem}[section]
\newtheorem{lemma}[theorem]{Lemma}
\newtheorem{Def}[theorem]{Definition}
\newtheorem{remark}[theorem]{Remark}
\newtheorem{Pro}{Proof}[section]
\newtheorem{corollary}[theorem]{Corollary}
\allowdisplaybreaks

\section{Introduction}
In this article, we establish existence, uniqueness and regularity results to the following degenerate singular elliptic boundary value problem:
 \begin{equation} \label{main}
\begin{gathered}
\begin{cases}
-\operatorname{div}\big(\mathcal A(x,\nabla u)\big)=\frac{f(x)}{u^\delta}\text{ in }\Omega, \\
 u>0\text{ in }\Omega,\\ u = 0\text{ on }\partial\Omega,
 \end{cases}
\end{gathered}
\end{equation}
where $\delta > 0$, $\Omega$ is a bounded smooth domain in $\mathbb{R}^N$ with $N \geq 2$ and $f$ is a nonnegative function belong to some Lebesgue space but not identically zero. The function $\mathcal{A} : \Omega\times\mathbb{R}^N\to\mathbb{R}^N$ is Carath\'eodory by which we mean
\begin{itemize}
\item the function $\mathcal{A}(\cdot,s)$ is measurable on $\Omega$ for every $s\in\mathbb{R}^N$, and
\item the function $\mathcal{A}(x,\cdot)$ is continuous on $\mathbb{R}^N$ for a.e. $x\in\Omega$.
\end{itemize}

Moreover, the following additional hypothesis on the function $\mathcal{A}$ will be imposed throughout the paper.
\begin{itemize}
\item[(H1)] For any $w$ belong to the Muckenhoupt class $A_p$ (defined in section \ref{prel}), 
\item[(H2)] (Growth) $|\mathcal A(x,\zeta)| \leq |\zeta|^{p-1}w(x)$, for a.e. $x\in\Omega$,\,\,$\forall$\,\,$\zeta\in\mathbb{R}^N$.
\item[(H3)] (Degeneracy) $\mathcal A(x,\zeta)\cdot\zeta \geq |\zeta|^pw(x)$, for a.e. $x \in \Omega$,\,\,$\forall$\,\,$\zeta\in\mathbb{R}^N$.
\item[(H4)] (Homogeneity) $\mathcal A(x,t\zeta) = t\,|t|^{p-2}\,\mathcal{A}(x,\zeta)$, for $t \in \mathbb{R}$, $t \neq 0$.
\item[(H5)] (Strong Monotonicity) For $\gamma = \text{max}\,\big\{p,2\big\}$,
$$\big<\mathcal A(x,\zeta_1)-\mathcal A(x,\zeta_2),\zeta_1-\zeta_2\big> \geq c\,|\zeta_1-\zeta_2|^\gamma\big\{\overline{\mathcal A}(x,\zeta_1,\zeta_2)\big\}^{1-\frac{\gamma}{p}}w(x),$$
for some positive constant $c$ where $\overline{\mathcal{A}}$ is defined as 
$$\overline{\mathcal A}(x,\zeta_1,\zeta_2): = \frac{1}{w(x)}\big(\big<\mathcal A(x,\zeta_1),\zeta_1\big> + \big<\mathcal A(x,\zeta_2),\zeta_2\big>\big).$$
\end{itemize}

A prototype of the equation (\ref{main}) is given by the following boundary value problem
\begin{equation}\label{genmodel}
\begin{gathered}
               \begin{cases}               
              Lu:=-\operatorname{div}\big(M(x)|\nabla u|^{p-2}\nabla u\big)=\frac{f(x)}{u^\delta}\,\,\text{ in }\,\,\Omega,\\             u>0\text{ in }\Omega,\\ u = 0  \text{ on } \partial\Omega,
             \end{cases}
\end{gathered}
\end{equation}
where $M(\cdot)$ is a continuous function with values in the set of $N\times N$ symmetric matrix satisfying
\begin{equation*}\label{M}
|M(x)\,\zeta|\leq\,w(x)|\zeta|,\,\,M(x)\,\zeta\cdot\zeta\geq\,w(x)|\zeta|^2,\text{ for a.e.}\,x\in\Omega,\,\,\forall\,\,\zeta\in\mathbb{R}^N.
\end{equation*}

In case of $M(x)=w(x)I$, where $I$ is the $N\times N$ identity matrix, the operator $L$ reduces to the weighted $p$-Laplace operator $\Delta_{p,w}$ defined by
$$
\Delta_{p,w}u:=\text{div}\big(w(x)|\nabla u|^{p-2}\nabla u\big).
$$

We observe that for $w=1$, $\Delta_{p,w}u=\Delta_p u$, which is the standard $p$-Laplace operator. For the constant weight $w$, singular problems of type (\ref{genmodel}) has been widely studied in the last three decades, see \cite{Boc2, DArcoya, KBal, Boc3, Boccardo, Canino2, Canino1, Canino3, VFelli, GRadha, Giacomoni, YHaitao, AMohammed} and the references therein. We would like to point out some historical developments made in this direction which are closely related to the problem (\ref{genmodel}). \\
The case of $M=I$ and $p=2$ with Dirichlet boundary condition is settled in the pioneering work of Crandall et al \cite{CTartar}, where the existence of a unique classical solution $u\in C^2(\Omega)\cap C(\overline{\Omega})$ to the equation (\ref{genmodel}) is proved for any $\delta>0$. This solution $u\in W_{0}^{1,2}(\Omega)$ if and only if $\delta<3$, and for $\delta>1$, $u$ does not belong to $C^1(\overline{\Omega})$ is proved by Lazer-McKenna \cite{LMckena} for a positive H$\ddot{\text{o}}$lder continuous data $f$.

Boccardo-Orsina \cite{Boccardo} studied the semilinear case $p=2$ for a constant weight function $w$ and nonnegative (not identically zero) data $f$ in some Lebesgue space to obtain existence and regularity results for any $\delta>0$ to the problem \eqref{genmodel}. De Cave \cite{Decave} generalised these results in the quasilinear case $1<p<N$. Further Canino et al \cite{Canino1} proved existence in addition to uniqueness results to the problem \eqref{genmodel} in the full range $1<p<\infty$ for $M=I$. In summary depending on $p$ and the nonlinearity $f$, authors in \cite{Boccardo, Canino1, Decave} proved existence of a solution $u\in W_{0}^{1,p}(\Omega)$ if $0<\delta<1$ and $u\in W^{1,p}_{loc}(\Omega)$ such that $u^\frac{\delta+p-1}{p}\in W_{0}^{1,p}(\Omega)$ (this was the meaning of $u=0$ on $\partial\Omega$) if $\delta\geq 1$ to the problem \eqref{genmodel}. Moreover, we emphasize that when $0<\delta<1$ under the assumption $f\in L^1(\Omega)$, authors in \cite{Boccardo} (for $p=2$) proved existence result in a larger Sobolev space than $W_{0}^{1,2}(\Omega)$, whereas if $2-\delta+\frac{\delta-1}{N}\leq p<N$, author in \cite{Decave} proved existence result in a larger Sobolev space than $W_{0}^{1,p}(\Omega)$ to the problem \eqref{genmodel}. When $f$ is a Radon measure, existence results to singular $p$-Laplace equations has also been investigated in the recent past and we refer the reader to De Cave et al \cite{Decave3}, De Cave-Oliva \cite{Decave2} and the references therein. 

In contrast to \cite{Boccardo, Canino1, Decave}, a natural question can be posed to say what happens to the equation \eqref{genmodel} in the presence of a nonconstant weight function $w$? Indeed, our main motive in this paper is to answer this question affirmatively by providing a certain class of weight function (which may vanish or blow up near the origin) to ensure existence, uniqueness and regularity results analogous to \cite{Boccardo, Canino1, Decave} for the more general weighted singular problem \eqref{main}.

We have started with choosing the weight function $w$ in the class of Muckenhoupt weight $A_p$ whose theory is well developed, see \cite{Chua, Drabek, EFabes, Juh, Tero, Muc, Ews}. Such class of weights was firstly introduced by Muckenhoupt \cite{Muc}, where the author proved these are the only class of weights such that the Hardy-Littlewood maximal operator is bounded from the weighted Lebesgue space into itself and thus plays a very significant role in harmonic analysis.

Due to the presence of the weight function solutions of \eqref{main} are investigated in a weighted Sobolev space (see section $2$ for definition). We mainly adapt the approximation approach introduced by the authors in \cite{Boccardo} along \cite{Canino1, Decave} although there are some difficulties we will face in our setting. To be more precise, by regularizing the right hand size of \eqref{main} we prove existence of a uniform positve and bounded solutions to the approximated problem \eqref{approximated problem}. But in contrast to \cite{Boccardo}, weak convergence is not enough to pass the limit in the approximated problem \eqref{approximated problem}. In this concern a gradient convergence theorem is proved by the author in \cite{Decave} which allows to pass the limit (see also \cite{Canino1}). Here, we establish a counterpart of gradient convergence theorem in our setting (see Theorem \ref{Gradient convergence theorem}) by following the idea from Boccardo-Murat \cite{Bocgrad} in order to pass the limit in the equation \eqref{approximated problem} and obtain our existence results. The availability of embedding results in the classical Sobolev space $W^{1,p}(\Omega)$ (see \cite{AmAr, Evans}) is one of the main ingredient in \cite{Boccardo, Canino1, Decave}. Such embeddings are not readily available in our setting which we establish here for a subclass of $A_p$ (see Theorem \ref{embedding theorem for $A_s$}). Then following the idea from \cite{Boccardo,Decave} choosing suitable test functions into the equation \eqref{approximated problem} along with an application of our embedding theorem we obtain regularity results depending on the summability of $f$. Finally, to obtain uniqueness results, we establish a variational inequality (see Lemma \ref{varineq}) further avoiding the use of boundary continuity of solutions to the regularized $p$-Laplace equations (see e.g., \cite{GLib, JSerrin, PTolks, Trudinger}) as implemented by the authors in \cite{Canino1}.\\

\textbf{Notations:}
Throughout the paper, the following notations will be used:
\begin{itemize}
\item $X:=W_{0}^{1,p}(\Omega,w)$.
\item $X^{*}:=\text{Dual space of }X$.
\item $||u||_X:=||u||_{1,p,w}$.
\item $c,c_i$, $i\in\mathbb{N}$ will denote constants whose values may vary depending on the situation from line to line or even in the same line. 
\item $|S|:=$ Lebesgue measure of a set $S$.
\item $T_{\eta}(s):=\text{min}\{\eta,s\}$ for $\eta>0$, $s\geq 0$.
\item $B(x,r):$ Ball of radius $r$ with center $x$. 
\end{itemize}
This paper is organized as follows:
In section \ref{prel}, we present some preliminary results. In section \ref{existreg}, existence and regularity results and in section \ref{uniq1}, uniqueness results are proved.
\section{Preliminaries}\label{prel}
In this section, we present some basic properties of $A_p$ weights and a brief literature of the corresponding weighted Sobolev space. For a more general theory we refer the reader to look at the nice surveys by Dr{\'a}bek et al \cite{Drabek}, Fabes et al \cite{EFabes}, Heinonen et al \cite{Juh} and Kilpe$\ddot{\text{l}}$ainen \cite{Tero}. 
\subsection{Muckenhoupt Weight}
\begin{Def}
Let $w$ be a locally integrable function in $\mathbb{R}^N$ such that $0<w<\infty$ a.e. in $\mathbb{R}^N$. Then for $1<p<\infty$, we say that $w$ belong to the Muckenhoupt class $A_p$ or $w$ is an $A_p$-weight, if there exists a positive constant $c_{p,w}$ (called the $A_p$ constant of $w$) depending only on $p$ and $w$ such that for all balls $B$ in $\mathbb{R}^N$,
$$
\Big(\frac{1}{|B|}\int_{B}w \,dx\Big)\Big(\frac{1}{|B|}\int_{B}w^{-\frac{1}{p-1}}\,dx\Big)^{p-1}\leq c_{p,w}.
$$
\end{Def}
\textbf{Example:}
\begin{itemize}
\item $w(x) = |x|^\alpha \in A_p$ if and only if $-N < \alpha < N(p-1)$, see \cite{Juh, Tero}.
\end{itemize} 

\begin{Def}(Weighted Sobolev Space:)
For any $w \in A_p$, define the weighted Sobolev space $W^{1,p}(\Omega,w)$ by
$$
W^{1,p}(\Omega,w)=\{u:\Omega\to\mathbb{R}\text{ measurable }:||u||_{1,p,w}<\infty\},
$$
where
\begin{equation}\label{norm1}
||u||_{1,p,w} = \Big(\int_{\Omega}|u(x)|^{p} w(x)\,dx\Big)^\frac{1}{p} + \Big(\int_{\Omega}|\nabla u|^{p} w(x)\,dx\Big)^\frac{1}{p}.
\end{equation}
\end{Def}
\begin{itemize}
\item Observe that if $0<c\leq w\leq d$ for some constants $c$ and $d$, the weighted Sobolev space $W^{1,p}(\Omega,w)$ becomes the classical Sobolev space $W^{1,p}(\Omega)$.
\item The fact $w\in A_p$ implies $w \in L^1_{loc}(\Omega)$ and hence $C_{c}^{\infty}(\Omega)\subset W^{1,p}(\Omega,w)$. Therefore we can introduce the space 
$$
W^{1,p}_{0}(\Omega,w) = \overline{\big(C_{c}^{\infty}(\Omega),||.||_{1,p,w}\big)}.
$$
\item Both the spaces $W^{1,p}(\Omega,w)$ and $W^{1,p}_{0}(\Omega,w)$ are uniformly convex Banach spaces with respect to the norm $||.||_{1,p,w}$, see \cite{Juh}.
\end{itemize}
\begin{Def}\label{local space}
We say that $u \in W^{1,p}_{loc}(\Omega,w)$ if and only if $u \in W^{1,p}(\Omega',w)$ for every $\Omega'\Subset\Omega$. 
\end{Def}
\begin{theorem}(Poincar\'e inequality \cite{Juh})\label{Poincare inequality}
For any $w \in A_p$, we have
\begin{equation*}
\int_{\Omega}|\phi|^{p} w(x) \,dx\leq c\int_{\Omega}|\nabla\phi|^{p} w(x) \,dx\;\forall\;\phi\in C_{c}^\infty(\Omega),
\end{equation*} 
for some positive constant $c$ independent of $\phi$.
\end{theorem}
Using Theorem \ref{Poincare inequality}, an equivalent norm to $(\ref{norm1})$ on the space $W^{1,p}_{0}(\Omega,w)$ can be defined by
\begin{equation}\label{norm2}
||u||_{1,p,w}=\Big(\int_{\Omega}|\nabla u(x)|^pw(x)\,dx\Big)^\frac{1}{p}.
\end{equation} 
\subsection{Embedding Theorems}
The following compactness result follows from Chua et al \cite{Chua}.  
\begin{theorem}(Theorem 2.2, \cite{Chua})\label{Embedding for $A_p$}
 Let $w \in A_p$ with $1 < p < \infty$, then the inclusion map 
$$W^{1,p}(\Omega,w)\hookrightarrow L^p(\Omega,w)$$
 is compact. 
\end{theorem}
For the rest of the paper, we assume the weight function $w\in A_s$ unless otherwise stated where $A_s$ is a subclass of $A_p$ given by 
$$
A_s := \Big\{w\in A_p: w^{-s}\in L^{1}(\Omega)\,\,\text{for some}\,\,s\in\big[\frac{1}{p-1},\infty\big)\cap\big(\frac{N}{p},\infty\big)\Big\}.
$$
For example, $w(x)=|x|^\alpha$ with $-\frac{N}{s}<\alpha<\frac{N}{s}$ belong to $A_s$ for any $s\in\big[\frac{1}{p-1},\infty\big)\cap\big(\frac{N}{p},\infty\big)$, provided $1<p<N$. This subclass allows one to shift from the weighted Sobolev space into the classical Sobolev space using the idea of \cite{Drabek}. Indeed, we prove the following embedding theorem.
\begin{theorem}\label{embedding theorem for $A_s$}(Embedding from weighted to classical Sobolev space)
\begin{itemize}
\item For any $w \in A_s$, we have the following continuous inclusion map
\[
    W^{1,p}(\Omega,w)\hookrightarrow W^{1,p_s}(\Omega)\hookrightarrow 
\begin{cases}
    L^q(\Omega),& \text{for } p_s\leq q\leq p_s^{*}, \text{in case of } 1\leq p_s<N, \\
    L^q(\Omega),& \text{for } 1\leq q< \infty, \text{in case of } p_s=N, \\
    C(\overline{\Omega}),& \text{in case of } p_s>N,
\end{cases}
\]
where $p_s = \frac{ps}{s+1} \in [1,p)$.
\item Moreover, the above embeddings are compact except for $q=p_s^{*}$ in case of $1\leq p_s<N$. 
\item The same result holds for the space $W_{0}^{1,p}(\Omega,w)$.
\end{itemize}
\end{theorem} 
\begin{proof}
Let $u \in W^{1,p}(\Omega,w)$. Since $\frac{p}{p_s} > 1$, using H$\ddot{\mbox{o}}$lder inequality with exponents $\frac{p}{p_s}$ and $(\frac{p}{p_s})' = s+1$, we obtain
\begin{align*}
\int_{\Omega}|u(x)|^{p_s}dx
&=\int_{\Omega}|u(x)|^{p_s}w(x)^\frac{p_s}{p}w(x)^{-\frac{p_s}{p}}\,dx\\
&\leq(\int_{\Omega}|u(x)|^pw(x)\,dx)^\frac{p_s}{p}(\int_{\Omega}w(x)^{-s}\,dx)^\frac{1}{s+1},\\
\end{align*}
which implies 
\begin{equation}\label{proof1}
||u||_{L^{p_s}(\Omega)} \leq (\int_{\Omega}w(x)^{-s}\,dx)^\frac{1}{ps}(\int_{\Omega}|u(x)|^pw(x)\,dx)^\frac{1}{p}.
\end{equation}
Replacing $u$ by $\nabla u$, similarly we obtain 
\begin{equation}\label{proof2}
||\nabla u||_{L^{p_s}(\Omega)}\leq(\int_{\Omega} w(x)^{-s}\,dx)^\frac{1}{ps}(\int_{\Omega}|\nabla u|^p w(x)\,dx)^\frac{1}{p}.
\end{equation}
Adding (\ref{proof1}) and (\ref{proof2}) we have
$$
||u||_{W^{1,p_s}(\Omega)} \leq ||w^{-s}||_{L^1(\Omega)}^\frac{1}{ps}||u||_{1,p,w}.
$$
Hence the embedding 
$$
W^{1,p}(\Omega,w)\hookrightarrow W^{1,p_s}(\Omega)
$$
is continuous. The rest of the proof follows from the classical Sobolev embedding theorem (Theorem 1.1.3 of Ambrosetti-Arcoya \cite{AmAr}).
\end{proof}
\begin{remark}\label{embedding needed for a-priori estimate}
Observe that the fact $s\in\big[\frac{1}{p-1},\infty\big)\cap\big(\frac{N}{p},\infty\big)$ implies that $p_s^{*}>p$. Therefore, by Theorem $\ref{embedding theorem for $A_s$}$ there exists a constant $q>p$ such that the inclusion 
$$
W^{1,p}(\Omega,w)\hookrightarrow L^q(\Omega)
$$
is continuous. The existence of such $q$ is an important tool to prove some a priori estimates later, see \cite{Drabek} for more applications.
\end{remark}
Now we state two important theorems on $\mathcal{A}$ superharmonic functions, for the definition of such functions we refer the reader to \cite{Juh}.
\begin{theorem}\label{smp1}(Theorem 7.12, \cite{Juh})
A nonconstant $\mathcal{A}$ superharmonic function cannot attain its infimum in $\Omega$.
\end{theorem}
\begin{theorem}\label{smp2}(Corollary 7.18, \cite{Juh})
If $u \in W^{1,p}_{loc}(\Omega,w)$ is a weak supersolution of the equation
$$
-\operatorname{div}\mathcal{A}(x,\nabla u) = 0
$$
in $\Omega$, i.e.
$$
\int_{\Omega}\mathcal{A }(x,\nabla u)\cdot\nabla\phi\,dx \geq 0
$$
whenever $\phi \in C^\infty_{c}(\Omega)$ is nonnegative, then there exists $\mathcal{A}$ superharmonic function $v$ such that $v = u$ a.e.
\end{theorem}
\begin{theorem}\label{bdythm}
Let $u\in W^{1,p}_{loc}(\Omega,w)$ be positive a.e. in $\Omega$ and $\alpha\geq 1$ be such that $u^\alpha\in X$. Then for every $\epsilon>0$, we have $(u-\epsilon)^{+}\in X$.
\end{theorem}
\begin{proof}
Since $u^\alpha\in X$, there exists a sequence of nonnegative functions $\{\phi_n\}\in C_c^{\infty}(\Omega)$ such that $\phi_n$ converges to $u^\alpha$ in the norm of $X$. Set
$$
v_n:=\Big(\phi_n^{\frac{1}{\alpha}}-\epsilon\Big)^{+}.
$$
Observe that, since $\alpha\geq 1$, one has
$$
||v_n||_{X}^p=\int_{\Omega}\,w(x)|\nabla v_n|^p\,dx\leq\int_\big{\{\phi_{n}>\epsilon^\alpha\big\}}\,w(x)\epsilon^{\alpha(\frac{1}{\alpha}-1)}|\nabla\phi_n|^\alpha\,dx\leq M,
$$
where $M$ is a constant independent of $n$, since $||\phi_n||_X\leq c$ for some positive constant $c$ independent of $n$. Therefore, the sequence $v_n$ is uniformly bounded in $X$ and by the reflexivity of $X$, it follows that $(u-\epsilon)^{+}\in X$. 
\end{proof}
\begin{theorem}\label{Gradient convergence theorem}(Gradient Convergence Theorem)
Given $n\in\mathbb{N}$ and $w\in A_s$, consider the following equation
\begin{align}\label{gradeqn}
-\operatorname{div}\big(\mathcal{A}(x,\nabla u_n)\big) = G_n\text{ in }\Omega.
\end{align}

Assume that $u_n \to u$ weakly in $W^{1,p}(\Omega,w)$. In addition, suppose $G_n$ satisfies
$$
|<G_n, \phi>| \leq C_K\,||\phi||_{L^\infty(\Omega)},
$$
for all $\phi\in C_{c}^{\infty}(\Omega)\,\,\mbox{with\,\,support}\,\phi\subset\,K$, where $C_K$ is a constant depending on $K$. Then, upto a subsequence $\nabla u_n \to \nabla u$ pointwise a.e. in $\Omega$.
\end{theorem}
\begin{proof}
In the unweighted case this theorem is proved in Theorem 2.1 of \cite{Bocgrad} and following the same arguments we present the proof in the weighted case as follows:

\textbf{Step 1.} Fix a compact set $K\subset\Omega$ and a function $\phi_K\in C_c^{\infty}(\Omega)$ such that $0\leq\phi_K \leq 1$ and $\phi_K\equiv 1$ on $K$. Define the truncated function
\[
    L_{\mu}(s) := \begin{cases}
        s, & \text{for } |s|\leq \mu,\\
        \mu\frac{s}{|s|}, & \text{for } |s|> \mu.\\
        \end{cases}
\]
Then $v_n=\phi_K\,L_{\mu}(u_n-u)\in W_{0}^{1,p}(\Omega,w)$ with compact support.
\begin{align*}
&\int_{\Omega}\phi_{K}\big\{\mathcal{A}(x,\nabla u_n)-\mathcal{A}(x,\nabla u)\big\}\cdot\nabla L_{\mu}(u_n-u)\,dx\\
&=<G_n,v_n>-\int_{\Omega}L_{\mu}(u_n-u)\mathcal{A}(x,\nabla u_n)\cdot\nabla\phi_K\,dx-\int_{\Omega}\phi_K\,\mathcal{A}(x,\nabla u)\cdot\nabla L_{\mu}(u_n-u)\,dx.
\end{align*}
Now,
\begin{align*}
I_n :&= \Big|\int_{\Omega}L_{\mu}(u_n-u)\mathcal{A}(x,\nabla u_n)\cdot\nabla\phi_K\,dx\Big|\\
&\leq ||\nabla\phi_K||_{L^\infty(\Omega)}\int_{K}\,w|u_n-u||\nabla u_n|^{p-1}\,dx\\
&\leq ||\nabla\phi_K||_{L^\infty(\Omega)}||u_n-u||_{L^p(\Omega,w)}||u_n||^{p-1}_{W^{1,p}(\Omega,w)}.
\end{align*} 
Since $u_n\to u$ weakly in $W^{1,p}(\Omega,w)$, by Theorem \ref{Embedding for $A_p$} the sequence $I_n$ converges to $0$ as $n\to\infty$. Moreover, since the sequence $L_{\mu}(u_n-u)\to 0$ weakly in $W^{1,p}(\Omega,w)$ as $n\to\infty$, one has the sequence
$$
J_n := \int_{\Omega}\phi_K\,\mathcal{A}(x,\nabla u)\cdot\nabla L_{\mu}(u_n-u)\,dx
$$
converges to $0$ as $n\to\infty$. Now, by the given condition we have $|<G_n,v_n>|\leq c_K\mu$.

\textbf{Step 2.} Fix $\theta\in(0,1)$ and define the sequence of function 
$$
e_n(x)=\big\{\mathcal{A}(x,\nabla u_n)-\mathcal{A}(x,\nabla u)\big\}\cdot\nabla(u_n-u)(x).
$$
We denote by
$$
S_n^{\mu}=\big\{x\in K:|u_n(x)-u(x)|\leq \mu\big\},\,G_n^{\mu}=\big\{x\in K:|u_n(x)-u(x)|>\mu\big\}.
$$
Therefore
$$
\int_{K}e_n^{\theta}\,dx=\int_{S_n^{\mu}}e_n^{\theta}\,dx+\int_{G_n^{\mu}}e_n^{\theta}\,dx\leq \Big(\int_{S_n^{\mu}}e_n\,dx\Big)^{\theta}|S_n^{\mu}|^{1-\theta}+\Big(\int_{G_n^{\mu}}e_n\,dx\Big)^{\theta}|G_n^{\mu}|^{1-\theta}.
$$
By Theorem \ref{embedding theorem for $A_s$}, we have $u_n\to u$ strongly in $L^{p_s}(\Omega)$. Therefore 
$|G_n^{\mu}|\to 0\text{ as }n\to\infty.$ Also, the sequence $\{e_n\}$ is bounded in $L^1(\Omega)$, since
\begin{align*}
\int_{\Omega}|e_n|\,dx&=\int_{\Omega}|\mathcal{A}(x,\nabla u_n)\cdot\nabla u_n-\mathcal{A}(x,\nabla u)\cdot\nabla u_n-\mathcal{A}(x,\nabla u_n)\cdot\nabla u+\mathcal{A}(x,\nabla u)\cdot\nabla u|\,dx\\
&\leq\int_{\Omega}\,w\{|\nabla u_n|^p+|\nabla u_n||\nabla u|^{p-1}+|\nabla u_n|^{p-1}|\nabla u|+|\nabla u|^p\}\,dx\\
&\leq M,
\end{align*}
for some constant $M$ independent of $n$. By Step $1$ and the fact $\phi_K\equiv 1$ on $K$, we obtain
$$
\limsup_{n\to\infty}\int_{K}e_n^{\theta}\,dx\leq (c_K\mu)^\theta|\Omega|^{1-\theta}.
$$
Letting $\mu\to 0$, we have $e_n^{\theta}\to 0$ in $L^1(K)$. Therefore upto a subsequence $e_n(x)\to 0$ a.e. in $\Omega$ and using the hypothesis (H5), we obtain upto a subsequence $\nabla u_{n}\to \nabla u$ pointwise a.e. in $\Omega$.
\end{proof}
Moreover, we will use the following three important results, see Ciarlet \cite{var} for Theorem \ref{MB}-Theorem \ref{varineq} and Kinderlehrer-Stampacchia \cite{Gstam} for Theorem \ref{useful for L infinity estimate} respectively.
\begin{theorem}\label{MB}(Theorem 9.14, \cite{var})
Let $V$ be a real reflexive Banach space and let $A:V\to V^{*}$ be a coercive and demicontinuous monotone operator. Then $A$ is surjective, i.e., given any $f\in V^{*}$, there exists $u\in V$ such that $A(u)=f$. If $A$ is strictly monotone, then $A$ is also injective.  
\end{theorem}
\begin{theorem}(\cite{var})\label{varineq}
Let $U$ be a nonempty closed and convex subset of a real separable reflexive Banach space and let $A:V\to V^{*}$ be a coercive and demicontinuous monotone operator. Then for every $f\in V^{*}$, there exists $u\in U$ such that $$\big<A(u),v-u\big>\geq \big<f,v-u\big>\text{ for all }v\in U.$$ Moreover, if $A$ is strictly monotone, then $u$ is unique.  
\end{theorem}
\begin{theorem}\label{useful for L infinity estimate}(Lemma B.1, \cite{Gstam})
Let $\phi(t),\,k_0\leq t < \infty,$ be nonnegative and nonincreasing such that
$$
\phi(h)\leq\Big[\frac{c}{(h-k)^l}\Big]\phi_k^m,\;\;h > k > k_0,
$$
where $c,l,m$ are positive constants with $m > 1$. Then
$
\phi(k_0+d) = 0,
$
where
$$
d^l = c\big[\phi(k_0)\big]^{m-1}2^\frac{lm}{m-1}.
$$
\end{theorem}

\section{Existence and regularity results}\label{existreg}
\begin{Def}
A function $u\in W^{1,p}_{loc}(\Omega,w)$ is said to be a weak solution of the problem (\ref{main}), if for every $K\Subset\Omega$, there exists a positive constant $c_K$ such that $u \geq c_K > 0$ in $K$ and for all $\phi\in C_{c}^1(\Omega)$, one has
\begin{equation}\label{weak solution of the main problem}
\begin{gathered}
\begin{cases}
\int_{\Omega} \mathcal{A}\big(x,\nabla u(x)\big)\cdot\nabla\phi(x)\,dx = \int_{\Omega}\frac{f(x)}{u^\delta}\phi(x)\,dx,\\
u > 0\text{ in }\Omega,\,\,u = 0  \text{ on } \partial\Omega,
\end{cases}
\end{gathered}
\end{equation}
\end{Def}
where by $u=0$ on $\partial\Omega$, we mean that for some $\alpha\geq 1$, the function $u^\alpha\in X$.

Our main existence and regularity results in this paper reads as follows:
\subsection{The case $0<\delta<1$}
\begin{theorem}\label{delta less}
For any $0<\delta<1$, the problem (\ref{main}) has at least one weak solution in $X$, if
\begin{enumerate}
\item[(a)] $f \in L^{m}(\Omega),\,m = \big(\frac{p_s^{*}}{1-\delta}\big)^{'}$, provided $1 \leq p_s < N$, or 
\item[(b)] $f \in L^{m}(\Omega)$ for some $m > 1$, provided $p_s = N$, or
\item[(c)] $f \in L^1(\Omega)$ for $p_s > N$.
\end{enumerate}
\end{theorem}
\begin{theorem}\label{regularity delta less}
Let $0<\delta<1$, then the solution $u$ given by Theorem \ref{delta less} satisfies the following properties:
\begin{enumerate}
\item[(a)] For $1 \leq p_s < N,$
\begin{enumerate}
\item[(i)] if $f \in L^{m}(\Omega)$ for some $m \in\big[(\frac{p_s^{*}}{1-\delta})^{'},\frac{p_s^{*}}{p_s^{*}-p}\big)$, then $u \in L^t(\Omega)$, $t = p_s^{*}\,\gamma$ where $\gamma = \frac{(\delta+p-1)m^{'}}{(pm^{'}-p_{s}^*)}$.
\item[(ii)] if $f \in L^m(\Omega)$ for some $m > \frac{p_s^{*}}{p_s^{*}-p}$, then $u \in L^\infty(\Omega)$.
\end{enumerate}
\item[(b)] Let $p_s = N$ and assume $q > p$. Then if $f \in L^m(\Omega)$ for some $m \in \big((\frac{q}{1-\delta})',\frac{q}{q-p}\big)$, we have $u \in L^t(\Omega)$, $t = p\,\gamma$ where $\gamma= \frac{pm'}{pm'-q}$.
\item[(c)] For $p_s > N$ and $f \in L^1(\Omega)$, we have $u \in L^\infty(\Omega)$.
\end{enumerate}
\end{theorem}
\subsection{The case $\delta=1$}
\begin{theorem}\label{delta equal}
For $\delta = 1$ with any $p_s$, the problem (\ref{main}) has at least one weak solution in $X$, provided $f\in L^1(\Omega)$.
\end{theorem}
\begin{theorem}\label{regularity delta equal}
Let $\delta=1$, then the solution $u$ given by Theorem \ref{delta equal} satisfies the following properties:
\begin{enumerate}
\item[(a)] For $1 \leq p_s < N$,
\begin{enumerate}
\item[(i)] if $f  \in L^{m}(\Omega)$ for some $m \in \big(1,\frac{p_s^{*}}{p_s^{*}-p}\big)$, then $u \in L^t(\Omega)$, $t = p_s^{*}\gamma$, where $\gamma = \frac{pm^{'}}{(pm^{'}-p_{s}^*)}$.
\item[(ii)] if $f \in L^m(\Omega)$ for some $m > \frac{p_s^{*}}{p_s^{*}-p}$, then $u \in L^\infty(\Omega)$.
\end{enumerate}
\item[(b)] Let $p_s = N$ and $q > p$. Then if $f \in L^m(\Omega)$ for some $m \in\big(1,\frac{q}{q-p}\big)$, we have $u \in L^t(\Omega)$, $t = q\,\gamma$, where $\gamma = \frac{pm^{'}}{pm^{'}-q}$.
\item[(c)] For $p_s > N$ and $f \in L^1(\Omega)$, we have $u \in L^\infty(\Omega)$.
\end{enumerate}
\end{theorem}
\subsection{The case $\delta>1$}
\begin{theorem}\label{delta greater}
For $\delta > 1$ with any $p_s$, the problem (\ref{main}) has at least one weak solution, say $u$ in $W^{1,p}_{loc}(\Omega,w)$ such that $u^\frac{\delta+p-1}{p} \in X$, provided $f\in L^1(\Omega)$.
\end{theorem}
\begin{theorem}\label{regularity delta greater}
Let $\delta>1,$ then the solution $u$ given by Theorem \ref{delta greater} satisfies the following properties:
\begin{enumerate}
\item[(a)] For $1 \leq p_s < N,$
\begin{enumerate}
\item[(i)] if $f \in L^{m}(\Omega)$ for some $m \in \big(1,\frac{p_s^{*}}{p_s^{*}-p}\big)$, then $u \in L^t(\Omega)$ where $t = p_s^{*}\,\gamma$, where $\gamma = \frac{(\delta+p-1)m'}{pm'-p_s^*}$.
\item[(ii)] if $f \in L^m(\Omega)$ some $m > \frac{p_s^{*}}{p_s^{*}-p}$, then $u \in L^\infty(\Omega)$.
\end{enumerate}
\item[(b)] Let $p_s = N$ and assume $q > p$. Then if $f \in L^m(\Omega)$ for some $m \in \big(1,\frac{q}{q-p}\big)$,  we have $u \in L^t(\Omega)$, $t = q\,\gamma$, where $\gamma = \frac{(\delta+p-1)m^{'}}{pm^{'}-q}$.
\item[(c)] For $p_s > N$ and $f \in L^1(\Omega)$, we have $u \in L^\infty(\Omega)$.
\end{enumerate}
\end{theorem}
\subsection{Preliminaries}
For $n \in \mathbb{N}$, define $f_{n}(x) := \text{min}\,\big\{f(x),n\big\}$ and consider for $\delta > 0$, the approximated problem
\begin{equation}\label{approximated problem}
\begin{cases}
\begin{gathered}
-\operatorname{div}\big({\mathcal{A}(x,\nabla u)}\big) = \frac{f_n(x)}{\big(u+\frac{1}{n}\big)^\delta}\;\;\mbox{in}\;\;\Omega,\\
u>0\text{ in }\Omega,\,\,u = 0\;\mbox{on}\;\partial\Omega.
\end{gathered}
\end{cases}
\end{equation}
\begin{Def}
A function $u\in X$ is said to be a weak solution of the problem (\ref{approximated problem}) if $u>0$ in $\Omega$ and for all $\phi \in X$, one has
\begin{equation}\label{weak solution of the approximated problem}
\begin{split}
\int_{\Omega} \mathcal{A}(x,\nabla u)\cdot\nabla\phi(x)\,dx &= \int_{\Omega}\frac{f_n(x)}{\big(u+\frac{1}{n}\big)^\delta}\phi(x)\,dx.
\end{split}
\end{equation}
\end{Def}
Define the operator $J:X \to X^*$ by 
$$
<J(u),\phi> := \int_{\Omega} \mathcal{A}(x,\nabla u)\cdot\nabla\phi\,dx,\text{ for all }\phi,\,u\in X.
$$
\begin{lemma}\label{Minty}
$J$ is a surjective and strictly monotone operator.
\end{lemma}
\begin{proof}
The proof follows applying Theorem \ref{MB}, since
\begin{enumerate}
\item{\textbf{Boundedness:}} Using the H$\ddot{\text{o}}$lder's inequality and hypothesis (H2) we obtain 
\begin{align*}
||J(u)||_{X^*} =
&\sup_{||\phi||_X\leq 1}\big|<J(u),\phi>\big|\\
&\leq\sup_{||\phi||_{X}\leq 1}\big|\int_{\Omega} \mathcal{A}(x,\nabla u)\cdot\nabla\phi\,dx\big|\\
&\leq\sup_{||\phi||_X\leq 1}\big|\int_{\Omega}\big(w^\frac{1}{p'}|\nabla u|^{p-1}\big)\big(w^\frac{1}{p}|\nabla\phi|\big)\,dx\big|\\
&\leq\;||u||_{X}^{p-1}.
\end{align*}
Hence $J$ is bounded.
\item{\textbf{Demicontinuity:}} Let $u_n\to u$ in the norm of $X$, then $w^\frac{1}{p}\nabla u_n\to w^\frac{1}{p}\nabla u$ in $L^p(\Omega)$. Therefore upto a subsequence $u_{n_k}$ of $u_n$, we have $\nabla u_{n_k}(x)\to\nabla u(x)$ pointwise for a.e. $x\in\Omega$. Since the function $\mathcal{A}(x,\cdot)$ is continuous in the second variable, we have 
$$
w(x)^{{-}\frac{1}{p}}\mathcal{A}\big(x,\nabla u_{n_k}(x)\big)\to w(x)^{-\frac{1}{p}}\mathcal{A}\big(x,\nabla u(x)\big)
$$ pointwise for a.e. $x\in\Omega.$ Now using the growth condition (H2), we obtain
\begin{align*}
||w^{-\frac{1}{p}}\mathcal{A}(x,\nabla u_{n_k})||^\frac{p}{p-1}_{L^\frac{p}{p-1}(\Omega)}
&=\int_{\Omega}w^{-\frac{1}{p-1}}(x)\big|\mathcal{A}\big(x,\nabla u_{n_k}(x)\big)\big|^\frac{p}{p-1}\,dx\\
&\leq\;\int_{\Omega}w^{-\frac{1}{p-1}}(x)w^\frac{p}{p-1}(x)|\nabla u_{n_k}(x)|^p\,dx\\
&\leq\;||u_{n_k}||_{X}^p\\
&\leq\; c^p
\end{align*}
where $||u_{n_k}||_{X}\leq c$. Therefore since the sequence $w^{-\frac{1}{p}}\mathcal{A}\big(x,\nabla u_{n_k}(x)\big)$ is uniformly bounded in $L^\frac{p}{p-1}(\Omega),$ we have $w^{-\frac{1}{p}}\mathcal{A}\big(x,\nabla u_{n_k}(x)\big)\to w^{-\frac{1}{p}}\mathcal{A}\big(x,\nabla u(x)\big)$ weakly in $L^\frac{p}{p-1}(\Omega)$, see Jakszto \cite{lpbound}. Since the weak limit is independent of the choice of the subsequence $u_{n_k},$ it follows that 
$$
w^{-\frac{1}{p}}\mathcal{A}\big(x,\nabla u_{n}(x)\big)\to w^{-\frac{1}{p}}\mathcal{A}\big(x,\nabla u(x)\big)
$$ weakly. Now $\phi\in X$ implies the function $w^\frac{1}{p}\nabla\phi\in L^p(\Omega)$ and therefore by the weak convergence, we obtain 
$$<J(u_n),\phi> \to <J(u),\phi>$$ 
$\mbox{as}\;n\to\infty$ and hence $J$ is demicontinuous.
\item{\textbf{Coercivity:}} Using (H3), we have the inequality
$$
<J(u),u> = \int_{\Omega} \mathcal{A} (x,\nabla u)\cdot\nabla u\,dx \geq \int_{\Omega}w|\nabla u|^{p}\,dx = ||u||^p_{X}.
$$ 
Therefore $J$ is coercive.
\item{\textbf{Strict monotonicity:}} Using the strong monotonicity condition (H5), for all $u \neq v\in X$, we have  
\begin{align*}
<J(u)-J(v),u-v>&=\int_{\Omega}\Big\{\mathcal{A}\big(x,\nabla u(x)\big)-\mathcal{A}\big(x,\nabla v(x)\big)\Big\}\cdot\nabla\big(u(x)-v(x)\big)dx>0.
\end{align*}
\end{enumerate}
\end{proof}
\begin{lemma}\label{inverse continuity}
The operator $J^{-1}: X^{*}\to X$ is bounded and continuous.
\end{lemma}
\begin{proof}
Using the H$\ddot{\text{o}}$lder's inequality for all $u,v\in X$, we have the estimate
\begin{equation}\label{estimate for the boundedness of the inverse of J}
\begin{split}
\big<J(v)-J(u),v-u\big>&\geq\big(||v||^{p-1}_X-||u||^{p-1}_X\big)\big(||v||_X-||u||_X\big),
\end{split}
\end{equation}
which implies the operator $J^{-1}$ is bounded. Suppose by contradiction $J^{-1}$ is not continuous, then there exists $g_k\to g$ strongly in $X^*$ such that $||J^{-1}(g_k)-J^{-1}(g)||_X\geq\gamma$ for some $\gamma>0.$ Denote by $u_k=J^{-1}(g_k)$ and $u=J^{-1}(g).$ Therefore, using (H3) we have
\begin{align*}
||u_k||_X^{p}&\leq\int_{\Omega}\mathcal{A}\big(x,\nabla u_k(x)\big)\cdot\nabla u_k(x)\,dx\\
&=<J(u_k),u_k>\\
&=<g_k,u_k>\\
&\leq||g_k||_{X^*}||u_k||_{X},
\end{align*}
which implies
$$
||u_k||_{X}^{p-1} \leq ||g_k||_{X^*}.
$$ 
Since $g_k \to g$ strongly in $X^*$, we have the sequence $\{u_k\}$ uniformly bounded in $X.$ Therefore upto subsequence there exists $u^1\in X$ such that $u_k\to u^1$ weakly in $X$. Now 
\begin{align*}
\big<J(u_k)-J(u^1),u_k-u^1\big>
&=\big<J(u_k)-J(u)+J(u)-J(u^1),u_k-u^1\big>\\
&=\big<J(u_k)-J(u),u_k-u^1\big>+\big<J(u)-J(u^1),u_k-u^1\big>.
\end{align*}
Since $J(u_k) \to J(u)$ in $X^*$ and $u_k\to u^1$ weakly in $X$, both the terms 
$$
\big<J(u_k)-J(u),u_k-u^1\big>\,\,\mbox{and}\,\,\big<J(u)-J(u^1),u_k-u^1\big>
$$
converges to $0$ as $k \to \infty$. Therefore,
$$
\big<J(u_k)-J(u^1),u_k-u^1\big> \to 0\;\;\mbox{as}\;\;k\to\infty.
$$
Putting $v = u_k$ and $u = u^1$ in the inequality (\ref{estimate for the boundedness of the inverse of J}) we obtain $||u_k||_X\to||u^1||_X.$ Therefore by the uniform convexity of $X,$ it follows that $u_k\to u^1$ in $X$ which together with the convergence $J(u_k)\to J(u)$ in $X^*$ implies that $J(u^1)=J(u)$. Now the injectivity of $J$ implies $u=u^1$, a contradiction to our assumption. Hence $J^{-1}$ is continuous.
\end{proof}
\begin{lemma}\label{lemma to prove the continuity of A}
Let $\zeta_k$, $\zeta \in X$ satisfies,
\begin{align*}
<J(\zeta_k),\phi> &= <h_k,\phi>,\\
<J(\zeta),\phi> &= <h,\phi>,
\end{align*}
for all $\phi \in X$ where $< , >$ denotes the dual product between $X^*$ and $X$. If $h_k\to h$ in $X^*$, then we have $\zeta_k\to\zeta$ in $X$.
\end{lemma}
\begin{proof}
By the strict monotonicity of $J$, we have $J(\zeta) = h$ and $J(\zeta_k) = h_k$. Therefore applying Lemma $\ref{inverse continuity}$, if $h_k\to h$ in $X^*$ then $J^{-1}(h_k)\to J^{-1}(h)$ i.e. $\zeta_k \to \zeta$ as $k \to \infty$. Hence the proof. 
\end{proof}
Using Lemma $\ref{Minty}$, we can define the operator $A : L^{p_s}(\Omega)\to X$ by $A(v)=u$ where $u\in X$ is the unique weak solution of the problem
\begin{equation}\label{fixed point problem}
\begin{split}
-\operatorname{div}\big(\mathcal{A}(x,\nabla u)\big) &= \frac{f_n(x)}{\big(|v|+\frac{1}{n}\big)^\delta}\;\mbox{in}\;\Omega,
\end{split}
\end{equation}
i.e., for all $\phi\in X$,
\begin{align*}
\int_{\Omega}\mathcal{A}\big(x,\nabla u(x)\big)\cdot\nabla\phi(x)\,dx&=\int_{\Omega}\frac{f_n(x)}{\big(|v(x)|+\frac{1}{n}\big)^\delta}\phi(x)\,dx.
\end{align*}
\begin{lemma}\label{continuity of A}
The map $A : L^{p_s}(\Omega)\to X$ is continuous as defined above. 
\end{lemma}
\begin{proof}
Let $v_k \to v$ in $L^{p_s}(\Omega)$. Suppose $A(v_k) = \zeta_k$ and $A(v) = \zeta$. Then for every fixed $n \in \mathbb{N}$ and for all $\phi\in X$, we have
\begin{equation*}
\begin{split}
\int_{\Omega}\mathcal{A}\big(x,\nabla \zeta_k(x)\big)\cdot\nabla\phi(x)\,dx &= \int_{\Omega}\frac{f_n(x)}{\big(|v_k(x)|+\frac{1}{n}\big)^\delta}\phi(x)\,dx\\
\int_{\Omega}\mathcal{A}\big(x,\nabla \zeta(x)\big)\cdot\nabla\phi(x)\,dx &= \int_{\Omega}\frac{f_n(x)}{\big(|v(x)|+\frac{1}{n}\big)^\delta}\phi(x)\,dx.
\end{split}
\end{equation*}
Denote by 
$$
g_k(x) = \frac{f_n(x)}{\big(|v_k(x)|+\frac{1}{n}\big)^\delta}\,\,\mbox{and}\,\,g(x) = \frac{f_n(x)}{\big(|v(x)|+\frac{1}{n}\big)^\delta}.
$$
Now, by Theorem \ref{embedding theorem for $A_s$}, one has
\begin{align*}
||g_k-g||_{X^{*}}=&\sup_{||\phi||_X\leq 1}\Big|\int_{\Omega}{f_n}\Big\{{\big(|v_k|+\frac{1}{n}\big)^{-\delta}}-{\big(|v|+\frac{1}{n}\big)^{-\delta}}\Big\}\phi\,dx\Big|\\
&\leq n ||\phi||_{L^{p_s}(\Omega)}\Big|\Big|{\big(|v_k|+\frac{1}{n}\big)^{-\delta}}-{\big(|v|+\frac{1}{n}\big)^{-\delta}}\Big|\Big|_{L^{p_s^{'}}(\Omega)}.
\end{align*}
Now since $\big|{\big(|v_k|+\frac{1}{n}\big)^{-\delta}}-{\big(|v|+\frac{1}{n}\big)^{-\delta}}\big|\leq 2n^{\delta+1}$ and $v_k\to v$ in $L^{p_s}(\Omega)$, upto a subsequence $v_{k_l} \to v$ pointwise a.e. in $\Omega$. As a consequence of the Lebesgue dominated theorem, we obtain $\Big|\Big|{\big(|v_{k_l}|+\frac{1}{n}\big)^{-\delta}}-{\big(|v|+\frac{1}{n}\big)^{-\delta}}\Big|\Big|_{L^{p_s^{'}}(\Omega)}\to 0$ as $k_l\to\infty$. Since the limit is independent of the choice of the subsequence, we have $\Big|\Big|{\big(|v_k|+\frac{1}{n}\big)^{-\delta}}-{\big(|v|+\frac{1}{n}\big)^{-\delta}}\Big|\Big|_{L^{p_s^{'}}(\Omega)}\to 0$ as $k\to\infty$. Therefore by Lemma $\ref{lemma to prove the continuity of A}$, we have $\zeta_k \to \zeta$ as $k \to \infty$. Hence $A: L^{p_s}(\Omega) \to X$ is a continuous map.
\end{proof}
\begin{theorem}\label{existence and uniqueness theorem for the approximated problem}
For any $p_s\geq 1$ the following holds:  
\begin{enumerate}
\item The problem (\ref{approximated problem}) has a unique weak solution, say $u_n$ in $X \cap L^{\infty}(\Omega)$ for every fixed $n \in \mathbb{N}$,
\item $u_{n+1}\geq u_n$ for every $n\in\mathbb{N}$, and
\item For every $K\Subset\Omega$ there exists a positive constant $C_K$ (independent of $n$) such that $u_n\geq C_K>0$ in $K$. 
\end{enumerate}
\end{theorem}
\begin{proof}
\begin{enumerate}
\item \textbf{Existence:} Define
$$
S := \big\{v \in L^{p_s}(\Omega) : \lambda\,A(v) = v,\;0 \leq \lambda \leq 1\big\}.
$$
Let $v_i \in S$ and $A(v_i) = u_i$ for $i = 1,2$. Using $u_i$ as test function in (\ref{fixed point problem}) we obtain
\begin{equation}\label{uniform estimate for fixed n}
||u_i||_X \leq c(n),
\end{equation} where $c(n)$ is a constant depending on $n$ but not on $u_i$, $i=1,2$. 
Therefore, by Lemma $\ref{continuity of A}$ and the compactness of the inclusion 
$$X\hookrightarrow L^{p_s}(\Omega)$$ together with the inequality $(\ref{uniform estimate for fixed n})$, it follows that the map
$$
A:L^{p_s}(\Omega)\to L^{p_s}(\Omega)\;\mbox{is\;both\;continuous\;and\;compact}.
$$
Observe that,
\begin{align*}
||v_1-v_2||_{L^{p_s}(\Omega)}
&=\lambda\,||A(v_1)-A(v_2)||_{L^{p_s}(\Omega)}\\
&=\lambda\,||u_1-u_2||_X\\
&\leq 2\,\lambda\,c(n)<\infty.
\end{align*}
Hence the set $S$ is bounded in $L^{p_s}(\Omega).$ By Schaefer's Fixed Point Theorem, there exists a fixed point of the map $A$, say $u_n$ i.e. $A(u_n) = u_n$ and hence $u_n \in X$ is a solution of (\ref{approximated problem}).\\\vspace{0.3cm}
{\textbf{$L^\infty$-estimate:}} For any $k >1$, define the set 
$$
A(k) := \big\{x\in\Omega:u_n(x)\geq k\mbox{ a.e. in }\Omega\big\}.
$$ 
Choosing 
\[
    \phi_{k}(x):= 
\begin{cases}
    u_n(x)-k,& \text{if } x\in A(k)\\
    0,              & \text{otherwise}
\end{cases}
\]
as a test function in (\ref{weak solution of the approximated problem}) together with the H$\ddot{\text{o}}$lder inequality and Remark $\ref{embedding needed for a-priori estimate}$, we obtain
\begin{align*}
\int_{\Omega}|\nabla\phi_k|^{p}w(x)\,dx&\leq n^{\delta+1}\int_{A(k)}|u_n(x)-k|\,dx\leq c\,n^{\delta+1}\,|A(k)|^\frac{q-1}{q}||\phi_k||_X. 
\end{align*}
Therefore we get
$$
||\phi_k||^{p-1}_X \leq c|A(k)|^\frac{q-1}{q},
$$
where $c$ depends on $n$. Now for $1 < k < h$, by the Remark $\ref{embedding needed for a-priori estimate}$, we obtain
\begin{align*}
(h-k)^p\,|A(h)|^\frac{p}{q}
&\leq\Big(\int_{A(h)}(u_n(x)-k)^q\,dx\Big)^\frac{p}{q}\\
&\leq\Big(\int_{A(k)}(u_n(x)-k)^q\,dx\Big)^\frac{p}{q}\\
&\leq\int_{\Omega}|\nabla\phi_k|^{p}\,w(x)\,dx\\
&\leq c|A(k)|^\frac{p'}{q'}.
\end{align*}
Hence we obtain the inequality
$$
|A(h)|\leq\frac{c}{(h-k)^q}|A(k)|^\frac{p'q}{pq'}.
$$
Now $q > p$ implies $\frac{p'q}{pq'} > 1$, therefore by Theorem \ref{useful for L infinity estimate}, we obtain 
$$
||u_n||_{L^\infty(\Omega)} \leq c,
$$ where $c$ is a constant dependent on $n$.
\item \textbf{Monotonicity:} Let $u_n$ and $u_{n+1}$ satisfies the equations
\begin{equation}\label{monotone1}
\begin{split}
\int_{\Omega}{\mathcal{A}\big(x,\nabla u_n(x)\big)}\cdot\nabla\phi(x)\,dx &= \int_{\Omega}\frac{f_n(x)}{\big(u_n+\frac{1}{n}\big)^\delta}\phi(x)\,dx
\end{split}
\end{equation}
and
\begin{equation}\label{monotone2}
\begin{split}
\int_{\Omega}{\mathcal{A}\big(x,\nabla u_{n+1}(x)\big)}\cdot\nabla\phi(x)\,dx &= \int_{\Omega}\frac{f_{n+1}(x)}{\big(u_{n+1}+\frac{1}{n+1}\big)^\delta}\phi(x)\,dx
\end{split}
\end{equation}
respectively for all $\phi \in X$.
Choosing $\phi = (u_n-u_{n+1})^+\in X$ and using the inequality $f_n(x) \leq f_{n+1}(x)$ we obtain after subtracting the equation (\ref{monotone1}) from (\ref{monotone2})
\begin{align*}
I
&:=\int_{\Omega}\Big\{\mathcal{A}\big(x,\nabla u_n(x)\big)-\mathcal{A}\big(x,\nabla u_{n+1}(x)\big)\Big\}\cdot\nabla(u_n-u_{n+1})^+(x)\,dx\\
&=\int_{\Omega}\Big\{\frac{f_n(x)}{\big(u_n(x)+\frac{1}{n}\big)^\delta}-\frac{f_{n+1}(x)}{\big(u_{n+1}(x)+\frac{1}{n+1}\big)^\delta}\Big\}(u_n-u_{n+1})^+(x)\,dx\\
&\leq\int_{\Omega}f_{n+1}(x)\Big\{\frac{1}{\big(u_n(x)+\frac{1}{n}\big)^\delta}-\frac{1}{\big(u_{n+1}(x)+\frac{1}{n+1}\big)^\delta}\Big\}(u_n-u_{n+1})^+(x)\,dx\\
&\leq 0.
\end{align*}
Again using the strong monotonicity condition (H5), we have
\begin{itemize}  
\item for $p\geq 2$, 
$$
0 \leq ||(u_n-u_{n+1})^+||_X^p \leq I \leq 0,
$$
\item and for $1 < p < 2$, 
$$
0 \leq \int_{\Omega} w(x)|\nabla(u_n-u_{n+1})^{+}|^{2}\{|\nabla u_n|^p+|\nabla u_{n+1}|^p\}^{1-\frac{2}{p}} \leq I \leq 0,
$$
which gives $u_{n+1} \geq u_n$.
\end{itemize}
\vspace{0.3cm}
{\textbf{Uniqueness:}} The uniqueness of $u_n$ follows by arguing similarly as in monotonicity.
\item Choosing $\phi=\text{min}\big\{u_n,0\big\}$ as a test function in the equation \eqref{weak solution of the approximated problem} we get $u_n\geq 0$ in $\Omega$. Applying Theorem \ref{smp1} we get $u_1>0$ in $\Omega$. Hence by the monotonicity and Theorem \ref{smp2} there exists $C_K>0$ (Independent of $n$) such that $u_n\geq C_K>0$ for every $K\Subset\Omega$.
\end{enumerate}
\end{proof}
\subsection{Proof of the existence and regularity results}
\subsection{The case $0<\delta<1$}
\begin{proof}[Proof of Theorem $\ref{delta less}$] Let $0<\delta<1$.
\begin{enumerate}
\item[(a)] Let $1 \leq p_s < N$. Choosing $\phi = u_n\in X$ as a test function in the equation (\ref{weak solution of the approximated problem}) and using H$\ddot{\text{o}}$lder's inequality together with the continuous embedding 
$
X \hookrightarrow L^{p_s^{*}}(\Omega)
$
we obtain 
\begin{align*}
||u_n||^p_{X}
&\leq\int_{\Omega}|f||u_n|^{1-\delta}\,dx\\
&\leq ||f||_{L^{m}(\Omega)}\big(\int_{\Omega}|u_n|^{(1-\delta){m}'\,dx}\big)^\frac{1}{{m}'}\\
&\leq c\,||f||_{L^{m}(\Omega)}||u_n||^{1-\delta}_{X}.
\end{align*}
Since $\delta+p-1 > 0$, we have
$
||u_n||_X \leq c,
$
where $c$ is a constant independent of $n$. Therefore one can apply Theorem $\ref{Gradient convergence theorem}$ to conclude that upto a subsequence $\nabla u_{n_k}\to\nabla u$ pointwise a.e. in $\Omega$. Since the function $\mathcal{A}(x,\cdot)$ is continuous, we have $w^{{-}\frac{1}{p}}(x)\,\mathcal{A}\big(x,\nabla u_{n_k}(x)\big) \to w^{-\frac{1}{p}}(x)\,\mathcal{A}\big(x,\nabla u(x)\big)$ pointwise for a.e. $x\in\Omega$. Now we observe that
\begin{align*}
||w^{-\frac{1}{p}}\mathcal{A}\big(x,\nabla u_{n_k}\big)||^\frac{p}{p-1}_{L^\frac{p}{p-1}(\Omega)}
&=\int_{\Omega}w^{-\frac{1}{p-1}}(x)\big|\mathcal{A}\big(x,\nabla u_{n_k}(x)\big)\big|^\frac{p}{p-1}dx\\
&\leq||u_{n_k}||_{X}^p\leq\,c^p.
\end{align*}
Since the sequence $w^{-\frac{1}{p}}\mathcal{A}(x,\nabla u_{n_k})$ is uniformly bounded in $L^\frac{p}{p-1}(\Omega)$, the sequence $w^{-\frac{1}{p}}\mathcal{A}\big(x,\nabla u_{n_k}(x)\big)\to w^{-\frac{1}{p}}\mathcal{A}\big(x,\nabla u(x)\big)$ weakly in $L^\frac{p}{p-1}(\Omega)$. As the weak limit is independent of the choice of the subsequence $u_{n_k},$ it follows that $w^{-\frac{1}{p}}\mathcal{A}\big(x,\nabla u_{n}(x)\big)\to w^{-\frac{1}{p}}\mathcal{A}\big(x,\nabla u(x)\big)$ weakly. Now $\phi\in X$ implies the function $w^\frac{1}{p}\nabla\phi\in L^p(\Omega)$ and hence by the weak convergence, we obtain
$$
\lim_{n \to \infty} \int_{\Omega} \mathcal{A}\big(x,\nabla u_{n}(x)\big)\cdot\nabla\phi(x)\,dx = \int_{\Omega}\mathcal{A}\big(x,\nabla u(x)\big)\cdot\nabla\phi(x)\,dx.
$$
Moreover, by Theorem \ref{existence and uniqueness theorem for the approximated problem} we have $u \geq u_n \geq c_K > 0$ for every $K\Subset\Omega$. Since for $\phi \in C_{c}^1(\Omega)$, one has 
$$
\Big|\frac{f_n\,\phi}{\big(u_n+\frac{1}{n}\big)^\delta}\Big| \leq \frac{||\phi||_{\infty}}{c_K^{\delta}}\,f \in L^1(\Omega),
$$
and $\frac{f_n}{\big(u_n+\frac{1}{n}\big)^\delta}\,\phi \to \frac{f}{u^\delta}\,\phi$ pointwise a.e. in $\Omega$ as $n \to \infty$, by the Lebesgue dominated convergence theorem we obtain
$$
\lim_{n \to \infty}\int_{\Omega}\,\frac{f_n}{\big(u_n+\frac{1}{n}\big)^\delta}\,\phi\,dx = \int_{\Omega}\,\frac{f}{u^\delta}\,\phi\,dx.
$$
Therefore we have for all $\phi \in C_c^{1}(\Omega)$, 
$$
\int_{\Omega}\mathcal{A}\big(x,\nabla u(x)\big)\cdot\nabla\phi(x)\,dx = \int_{\Omega}\,\frac{f}{u^\delta}\,\phi\,dx
$$
and hence $u\in X$ is a weak solution of (\ref{main}).
\item[(b)] Let $p_s = N$. Choosing $\phi=u_n\in X$ as a test function in (\ref{weak solution of the approximated problem}) and using H$\ddot{\text{o}}$lder inequality together with the continuous embedding
$
X\hookrightarrow L^{q}(\Omega),\,\,q\in[1,\infty),
$
we obtain
\begin{align*}
||u_n||^p_{X}
&\leq\int_{\Omega}|f||u_n|^{1-\delta}\,dx\\
&\leq ||f||_{L^{m}(\Omega)}\big(\int_{\Omega}|u_n|^{(1-\delta)m'}\,dx\big)^\frac{1}{m'}\\
&\leq c\,||f||_{L^{m}(\Omega)}||u_n||^{1-\delta}_{X},
\end{align*}
where $c$ is a constant independent of $n$. Since $\delta+p-1 > 0$ we have the sequence $\{u_n\}$ is uniformly bounded in $X$. Now arguing similarly as in case (a) we obtain the required result.
\item[(c)] Let $p_s > N$. Choosing $\phi = u_n\in X$ as a test function in (\ref{weak solution of the approximated problem}) and using H$\ddot{\text{o}}$lder inequality together with the continuous embedding
$
X\hookrightarrow L^\infty(\Omega)
$
we obtain
\begin{align*}
||u_n||^p_{X}
&\leq\int_{\Omega}|f||u_n|^{1-\delta}\,dx\\
&\leq ||f||_{L^1(\Omega)}||u_n||_{L^\infty(\Omega)}^{(1-\delta)}\\
&\leq c||f||_{L^{1}(\Omega)}||u_n||^{1-\delta}_{X}.
\end{align*}
Since $\delta+p-1 > 0$, we have
$
||u_n||_{X} \leq c,
$
where $c$ is a constant independent of $n$. Therefore the sequence $\{u_n\}$ is uniformly bounded in $X$. Arguing similarly as in (a) we obtain the required result.
\end{enumerate}
\end{proof}
\begin{proof}[Proof of Theorem \ref{regularity delta less}]
\begin{enumerate}
\item[(a)] Let $1 \leq p_s < N$, then $p_s^{*} > p$.
\begin{enumerate}
\item[(i)] We observe that 
\begin{itemize}
\item for $m = \big(\frac{p_s^{*}}{1-\delta}\big)'$ i.e., $(1-\delta)m' = p_s^{*}$, we have $\gamma = \frac{(\delta+p-1)m^{'}}{(pm^{'}-p_{s}^*)} = 1$ and
\item $m \in \big((\frac{p_s^{*}}{1-\delta})^{'},\frac{p_s^{*}}{p_s^{*}-p}\big)$ implies $\gamma = \frac{(\delta+p-1)m^{'}}{(pm^{'}-p_{s}^*)} > 1$. 
\end{itemize}
Note that $(p\gamma-p+1-\delta)m' = p_s^{*}\gamma$ and choosing $\phi = u_n^{p\gamma-p+1}\in X$ as a test function in $(\ref{weak solution of the approximated problem})$ we obtain
\begin{align*}
||u_n^{\gamma}||_X^{p}
&\leq ||f||_{L^m(\Omega)}\Big(\int_{\Omega}|u_n|^{p_s^{*}\gamma}\,dx\Big)^\frac{1}{m'}.
\end{align*}
Now using the continuous embedding $X \hookrightarrow L^{p_s^{*}}(\Omega)$ and the fact $\frac{p}{p_s^{*}}-\frac{1}{m'} > 0$ we obtain
$
||u_n^\gamma||_{L^{p_s^{*}}(\Omega)} \leq c,
$
where $c$ is independent of $n$ implies the sequence $\{u_n^\gamma\}$ is uniformly bounded in $L^t(\Omega)$ where $t = p_s^{*}\gamma$. Therefore the pointwise limit $u$ belong to $L^t(\Omega)$ e.g., see \cite{lpbound}.
\item[(ii)] Let $m > \frac{p_s^{*}}{p_s^{*}-p}$ and for $k > 1$, choosing $\phi_k = (u_n-k)^+ \in X$ as a test function in (\ref{weak solution of the approximated problem}) we obtain after using H$\ddot{\text{o}}$lder's and Young's inequality with $\epsilon\in(0,1)$
\begin{align*}
\int_{\Omega}w|\nabla\phi_k|^p\,dx
&\leq\,c\int_{A(k)}|f||u_n-k|\,dx\\
&\leq\,c\Big(\int_{A(k)}|f|^{p_s^{*}{'}}\,dx\Big)^\frac{1}{p_s^{*}{'}}
\Big(\int_{A(k)}|u_n-k|^{p_s^{*}} \,dx\Big)^\frac{1}{p_s^{*}}\\
&\leq\,c\Big(\int_{A(k)}|f|^{p_s^{*}{'}}\,dx\Big)^\frac{1}{p_s^{*}{'}}
\Big(\int_{\Omega}w|\nabla\phi_k|^p\,dx\Big)^\frac{1}{p}\\
&\leq\,c_\epsilon\Big(\int_{A(k)}|f|^{p_s^{*}{'}}\,dx\Big)^\frac{p'}{p_s^{*}{'}} + \epsilon\Big(\int_{\Omega}w|\nabla\phi_k|^p\,dx\Big),
\end{align*}
where $A(k)=\big\{x\in\Omega:u_n\geq k \text{ a.e. in }\Omega\big\}$. Since $m > \frac{p_s^{*}}{p_s^{*}-p}$, we have $m > p_s^{*}{'}$. Using H$\ddot{\text{o}}$lder's inequality in the above estimate we obtain
$$
\int_{\Omega}w|\nabla\phi_k|^p\,dx \leq c\,||f||^{p'}_{L^m(\Omega)}|A(k)|^{\frac{p'}{p_s^{*}{'}}\frac{1}{(\frac{m}{p_s^{*}{'}})^{'}}}
$$
where $c$ is a constant independent of $n$. Now using the continuous embedding
$
X\hookrightarrow L^{p_s^{*}}(\Omega)
$
we obtain for $1 < k < h$,
\begin{align*}
(h-k)^p|A(h)|^\frac{p}{p_s^{*}}
&\leq\Big(\int_{A(h)}(u-k)^{p_s^{*}}\,dx\Big)^\frac{p}{p_s^{*}}\\
&\leq\Big(\int_{A(k)}(u-k)^{p_s^{*}}\,dx\Big)^\frac{p}{p_s^{*}}\\
&\leq c\int_{\Omega}w|\nabla\phi_k|^p\,dx\\
&\leq c\,||f||^{p'}_{L^m(\Omega)}|A(k)|^{\frac{p'}{p_s^{*}{'}}\frac{1}{(\frac{m}{p_s^{*}{'}})^{'}}}.
\end{align*}
Therefore
$$
|A(h)| \leq\frac{c||f||_{L^m(\Omega)}^{\frac{p_s^{*}}{p-1}}}{(h-k)^{p_s^{*}}}|A(k)|^{\frac{p'p_s^{*}}{pp_s^{*}{'}}\frac{1}{(\frac{m}{p_s^{*}{'}})^{'}}}.
$$
Since ${\frac{p'p_s^{*}}{pp_s^{*}{'}}\frac{1}{(\frac{m}{p_s^{*}{'}})^{'}}} > 1$, by Theorem \ref{useful for L infinity estimate}, we have  
$
||u_n||_{L^\infty(\Omega)} \leq c,
$
where $c$ is a constant independent of $n$. Therefore we have $u \in L^\infty(\Omega)$.
\end{enumerate}
\item[(b)] Let $p_s = N$ and $q > p$. Observe that 
\begin{itemize}
\item for $m = \big(\frac{q}{1-\delta}\big)'$ i.e., $(1-\delta)m' = q$, we have $\gamma = \frac{(\delta+p-1)m^{'}}{(pm^{'}-q)} = 1$ and
\item $m \in \big((\frac{q}{1-\delta})^{'},\frac{pm'}{pm'-q}\big)$ implies $\gamma = \frac{(\delta+p-1)m^{'}}{(pm^{'}-q)} > 1$. 
\end{itemize}
Note that $(p\gamma-p+1-\delta)m' = q\,\gamma$ and choosing $\phi = u_n^{p\gamma-p+1}\in X$ as a test function in $(\ref{weak solution of the approximated problem})$ we obtain
\begin{align*}
||u_n^{\gamma}||_X^{p}
&\leq ||f||_{L^m(\Omega)}\Big(\int_{\Omega}|u_n|^{q\,\gamma}\,dx\Big)^\frac{1}{m'}.
\end{align*}
Now using the continuous embedding $X \hookrightarrow L^{q}(\Omega)$ and the fact $\frac{p}{q}-\frac{1}{m'} > 0$ we obtain
$
||u_n^\gamma||_{L^{q}(\Omega)} \leq c,
$
where $c$ is independent of $n$ implies the sequence $\{u_n^\gamma\}$ is uniformly bounded in $L^t(\Omega)$ where $t = q\,\gamma$. Therefore $u$ belong to $L^t(\Omega)$.
\item[(c)] Follows from Theorem $\ref{delta less}$ using the continuous embedding $X \hookrightarrow L^\infty(\Omega)$. 
\end{enumerate}
\end{proof}
\subsection{The case $\delta=1$}
\begin{proof}[Proof of Theorem $\ref{delta equal}$]
Let $\delta=1$ and $f\in L^1(\Omega)$. Then choosing $\phi = u_n\in X$ as a test function in $(\ref{weak solution of the approximated problem})$ for any $p_s\geq 1$, we obtain 
$
||u_n||^p_{X}\leq||f||_{L^1(\Omega)}.
$ Now arguing similarly as in Theorem \ref{delta less} we obtain the existence of weak solution $u\in X$ of (\ref{main}).
\end{proof}
\begin{proof}[Proof of Theorem \ref{regularity delta equal}]
\begin{enumerate}
\item[(a)] Let $1 \leq p_s <N$, then $p_s^{*} > p$.
\begin{enumerate}
\item[(i)] Observe that $m \in \big(1,\frac{p_s^{*}}{p_s^{*}-p}\big)$ implies $\gamma = \frac{pm^{'}}{(pm^{'}-p_{s}^*)} > 1$. Now choosing $\phi = u_n^{p\gamma-p+1} \in X$ as a test function in $(\ref{weak solution of the approximated problem})$ together with the continuous embedding $X\hookrightarrow L^{p_s^{*}}(\Omega)$ and arguing similarly as in part (i) of Theorem $\ref{regularity delta less}$ we obtain the required result.
\item[(ii)] Follows arguing similarly as in part (ii) of Theorem $\ref{regularity delta less}$.
\end{enumerate}
\item[(b)] Let $p_s = N$ and $q>p$. Observe that $m \in \big(1,\frac{q}{q-p}\big)$ implies $\gamma = \frac{pm'}{pm'-q} > 1$. Choosing $\phi = u_n^{p\gamma-p+1} \in X$ as a test function in $(\ref{weak solution of the approximated problem})$ together with the continuous embedding $X \hookrightarrow L^q(\Omega)$ and proceeding similarly as in part (b) of Theorem $\ref{regularity delta less}$ we obtain the required result.
\item[(c)] Follows from Theorem $\ref{delta equal}$ using the continuous embedding $X \hookrightarrow L^\infty(\Omega)$.
\end{enumerate}  
\end{proof}
\subsection{The case $\delta>1$}
\begin{proof}[Proof of Theorem $\ref{delta greater}$] Let $\delta > 1$ and $f \in L^1(\Omega)$ with $p_s\geq 1$. By Theorem \ref{existence and uniqueness theorem for the approximated problem} for every fixed $n\in\mathbb{N}$ we have $u_n\in L^\infty(\Omega)$ (the bound may depend on $n$). Choosing $\phi = u_n^{\delta} \in X$ as a test function in $(\ref{weak solution of the approximated problem})$ (which is admissible since $\delta > 1$ and $u_n \in L^\infty(\Omega)$ by Theorem \ref{existence and uniqueness theorem for the approximated problem}) we obtain
\begin{align*}
\int_{\Omega}\delta u_n^{\delta-1}|\nabla u_n|^pw(x)\,dx\leq\int_{\Omega}\delta u_n^{\delta-1} \mathcal{A}(x,\nabla u_n)\cdot\nabla u_n\,dx
&\leq\int_{\Omega}|f(x)|\,dx,
\end{align*}
which implies
$$
\int_{\Omega}w\Big|\nabla\big(u_n^{\frac{\delta+p-1}{p}}\big)\Big|^p\,dx\leq c\,||f||_{L^1(\Omega)},
$$
where $c$ is independent of $n$. Therefore the sequence $\Big\{u_n^{\frac{\delta+p-1}{p}}\Big\}$ is uniformly bounded in $X$. 
Let $\phi\in C_c^{\infty}(\Omega)$ and consider $v_n=\phi^{p}\,u_n\in X$. We observe that 
\begin{equation}\label{delta greater estimate one}
\begin{split}
\int_{\Omega}\mathcal{A}(x,\nabla u_n)\cdot\nabla(\phi^{p}u_n)\,dx &= p\,\int_{\Omega}\phi^{p-1}\,u_n\,\mathcal{A}(x,\nabla u_n)\cdot\nabla\phi\,dx +\int_{\Omega}\phi^{p}\,\mathcal{A}(x,\nabla u_n)\cdot\nabla u_n\,dx,
\end{split}
\end{equation}
and using Young's inequality for $\epsilon\in(0,1)$, we obtain for some positive constant $c_{\epsilon}$,
\begin{equation}\label{delta greater estimate two}
\begin{split}
\big|p\,\int_{\Omega}\phi^{p-1}u_n\,\mathcal{A}(x,\nabla u_n)\cdot\nabla\phi\,dx\big|&\leq \epsilon\,\int_{\Omega}\,w|\phi|^p|\nabla u_n|^p\,dx + c_\epsilon\,\int_{\Omega}\,w\,|u_n|^p\,|\nabla\phi|^p\,dx.
\end{split}
\end{equation}
Now choosing $\phi = v_n \in X$ as a test function in (\ref{weak solution of the approximated problem}) and using the estimates (\ref{delta greater estimate one}), (\ref{delta greater estimate two}), we obtain
\begin{align*}
&\int_{\Omega}\phi^p|\nabla u_n|^p w(x)\,dx\\
&\leq\int_{\Omega}\phi^{p}\mathcal{A}(x,\nabla u_n)\cdot\nabla u_n\,dx\\
&=\int_{\Omega}\frac{f_n}{\big(u_n+\frac{1}{n}\big)^\delta}\phi^{p}u_n\,dx-p\int_{\Omega}\phi^{p-1}u_n \mathcal{A}(x,\nabla u_n)\cdot\nabla\phi\,dx\\
&\leq\int_{K}\frac{f_n}{u_n^{\delta}}\phi^{p}\,dx + \epsilon\int_{\Omega}|\phi|^p|\nabla u_n|^p w(x)\, dx + c_\epsilon\int_{\Omega}|u_n|^{p}|\nabla\phi|^{p}w(x)\,dx\\
&\leq\frac{||\phi||_{L^\infty(\Omega)}}{c_K^{\delta}}||f||_{L^1(\Omega)} + \epsilon\int_{\Omega}|\phi|^p|\nabla u_n|^p w(x)\,dx + c_\epsilon||\nabla\phi||^p_{L^\infty(\Omega)}\int_{K}\frac{1}{u_n^{\delta-1}}w\big|u_{n}^\frac{\delta+p-1}{p}\big|^p\,dx\\
&\leq c_{\phi}||f||_{L^1(\Omega)} + \epsilon\int_{\Omega}|\phi|^p|\nabla u_n|^p w(x)\,dx + c_{\phi}\big|\big|u_{n}^\frac{\delta+p-1}{p}\big|\big|_X,
\end{align*}
where $K$ is the support of $\phi$ and $c_\phi$ is a constant depending on $\phi$. Therefore we have
$$
(1-\epsilon)\int_{\Omega}\phi^p|\nabla u_n|^pw(x)\,dx\leq c_{\phi}\Big\{||f||_{L^1(\Omega)} + \big|\big|u_{n}^\frac{\delta+p-1}{p}\big|\big|_X\Big\}. 
$$
Now since the sequence $\Big\{u_{n}^\frac{\delta+p-1}{p}\Big\}$ is uniformly bounded in $X$ we have the sequence $\{u_n\}$ is uniformly bounded in $W^{1,p}_{loc}(\Omega,w)$. Now arguing similarly as in Theorem \ref{delta less}, we obtain $u\in W^{1,p}_{loc}(\Omega,w)$ is a weak solution of (\ref{main}). The fact that $u^\frac{\delta+p-1}{p}\in X$ follows from the uniform boundedness of the sequence $\Big\{u_n^\frac{\delta+p-1}{p}\Big\}$ in $X$. 
\end{proof}
\begin{proof}[Proof of Theorem \ref{regularity delta greater}]
\begin{enumerate}
\item[(a)] Let $1 \leq p_s <N$, then $p_s^{*} > p$.
\begin{enumerate}
\item[(i)] Observe that $m \in \big(1,\frac{p_s^{*}}{p_s^{*}-p}\big)$ implies $\gamma = \frac{(\delta+p-1)m'}{pm'-p_s^*} > \frac{\delta+p-1}{p} > 1$, since $\delta > 1$. Now choosing $\phi = u_n^{p\gamma-p+1} \in X$ as a test function in $(\ref{weak solution of the approximated problem})$ together with the continuous embedding $X\hookrightarrow L^{p_s^{*}}(\Omega)$ and arguing similarly as in part (i) of Theorem $\ref{regularity delta less}$ the result follows.
\item[(ii)] Follows by arguing similarly as in part (ii) of Theorem $\ref{regularity delta less}$.
\end{enumerate}
\item[(b)] Let $p_s = N$ and $q>p$. Observe that $\delta > 1$, $m \in \big(1,\frac{q}{q-p}\big)$ implies $\gamma = \frac{(\delta+p-1)m'}{pm'-q} > 1$. Choosing $\phi = u_n^{p\gamma-p+1} \in X$ as a test function in $(\ref{weak solution of the approximated problem})$ together with the continuous embedding $X \hookrightarrow L^q(\Omega)$ and proceeding similarly as in part (b) of Theorem $\ref{regularity delta less}$ we obtain the required result.
\item[(c)] Follows from Theorem $\ref{delta greater}$ using the continuous embedding $X \hookrightarrow L^\infty(\Omega)$.
\end{enumerate}   
\end{proof}
\section{Uniqueness results}\label{uniq1}
In this section we state and prove our main uniqueness results.
\subsection{The case $0<\delta\leq 1$}
\begin{theorem}\label{uniquenss delta less than one}
For any $0<\delta\leq 1$ and $w \in A_p$, the problem (\ref{main}) admits at most one weak solution in $W_{0}^{1,p}(\Omega,w)$ for any non-negative $f\in L^1(\Omega)$.
\end{theorem}
\begin{proof}Let $0<\delta\leq 1$, $w \in A_p$ be arbitrary and $u_1$, $u_2\in X$ are two solutions of the equation $(\ref{main})$. The fact $(u_1-u_2)^+ \in X$ allows us to choose $\{\varphi_{n}\} \in C_{c}^\infty(\Omega)$ converging to $(u_1-u_2)^+$ in $||\cdot||_X$. Now setting, $\psi_n := \text{min}\,\big\{(u_1-u_2)^+,\varphi^{+}_{n}\big\} \in X\cap L^\infty_{c}(\Omega)$ as a test function in (\ref{main}) we get
$$
\int_{\Omega}\big(\mathcal{A}(x,\nabla u_1)-\mathcal{A}(x,\nabla u_2)\big)\cdot\nabla\psi_n\,dx \leq \int_{\Omega}f\Big(\frac{1}{u_{1}^\delta}-\frac{1}{u_{2}^\delta}\Big)\psi_n\,dx \leq 0.
$$
Passing to the limit and using the strong monotonicity condition (H5), $(u_{1}-u_{2})^+ = 0$ a.e. in $\Omega$ which implies $u_1 \leq u_2$.
Similarly changing the role of $u_1$ and $u_2$, we get $u_2 \leq u_1$. Therefore, $u_1 \equiv u_2$.
\end{proof}
\subsection{The case $\delta>1$}
\begin{theorem}\label{uniquenss delta greater than one}
Let $\delta>1$ and $w\in A_s$. Then the problem (\ref{main}) has at most one weak solution in $W^{1,p}_{loc}(\Omega,w)$ if 
\begin{itemize}
\item[(a)] $f\in L^m(\Omega)$ for some $m=(p_s^{*})'$, provided $1\leq p_s<N$, or
\item[(b)] $f\in L^m(\Omega)$ for some $m>1$, provided $p_s=N$, or
\item[(c)] $f\in L^1(\Omega)$ for $p_s>N$.
\end{itemize}
\end{theorem}
\begin{remark}\label{Improvement}
In case $w\equiv 1$ our main results in this paper will hold by replacing $p_s$ by $p$. Moreover, since $(p^{*})'<\frac{N}{p}$, Theorem \ref{uniquenss delta greater than one} improves the range of $f$ in Theorem 1.5 of \cite{Canino1} to get the uniqueness provided $1<p<N$.
\end{remark}
\textbf{Preliminaries:}
Define for $k>0$ and $\delta>1$, the truncated function
\begin{equation*}
  g_k(s) := 
  \begin{cases}
    \text{min}\big\{s^{-\delta},k\big\}, & \text{for } s>0,\\
    k, & \text{for } s\leq 0.
  \end{cases}
\end{equation*}
\begin{Def}
We say that $v(>0)\in W_{loc}^{1,p}(\Omega,w)$ is a super-solution of the problem (\ref{main}), if for every $K\Subset\Omega$, there exists a positive constant $c_K$ such that $v\geq c_K>0$ in $K$ and for every non-negative $\phi\in C_c^{1}(\Omega)$, one has
\begin{equation}\label{supsolution}
\int_{\Omega}\mathcal{A}(x,\nabla v)\cdot\nabla\phi\,dx\geq\int_{\Omega}\frac{f(x)}{v^\delta}\phi\,dx.
\end{equation}
\end{Def}
\begin{Def}
We say that $v(>0)\in W_{loc}^{1,p}(\Omega,w)$ is a sub-solution of the problem (\ref{main}), if for every $K\Subset\Omega$, there exists a positive constant $c_K$ such that $v\geq c_K>0$ in $K$ and for every non-negative $\phi\in C_c^{1}(\Omega)$, one has
\begin{equation}\label{subsolution}
\int_{\Omega}\mathcal{A}(x,\nabla v)\cdot\nabla\phi\,dx\leq\int_{\Omega}\frac{f(x)}{v^\delta}\phi\,dx.
\end{equation}
\end{Def}
For a fixed super-solution $v$ of (\ref{main}), consider the following non-empty closed and convex set
$$
\mathbb{K} := \big\{\phi\in X:0\leq\phi\leq v\text{ a.e. in }\Omega\big\}.
$$
\begin{lemma}\label{unilemma1}
There exists $z\in\mathbb{K}$ such that for every non-negative $\phi\in C_c^{1}(\Omega)$, one has
\begin{equation}\label{ineq1}
\int_{\Omega}\mathcal{A}(x,\nabla z)\cdot\nabla\phi\,dx\geq\int_{\Omega}f(x)g_k(z)\phi\,dx.
\end{equation}
\end{lemma}
\begin{proof}
Under the assumptions on $f$ and applying Theorem \ref{embedding theorem for $A_s$} one can define the operator $J_k:X\to X^{*}$ for every $u,\psi\in X$ by $$<J_k(u),\psi> := \int_{\Omega}\mathcal{A}(x,\nabla u)\cdot\nabla\psi\,dx-\int_{\Omega}fg_k(u)\psi\,dx.$$
Following the same arguments as in Lemma \ref{Minty}, it follows that $J_k$ is demicontinuous, coercive and strictly monotone. As a consequence of Theorem \ref{varineq}, there exists a unique $z\in \mathbb{K}$ such that for every $\psi\in\mathbb{K}$, one has 
\begin{equation}\label{ineq2}
\int_{\Omega}\mathcal{A}(x,\nabla z)\cdot\nabla(\psi-z)\,dx\geq\int_{\Omega}f(x)g_k(z)(\psi-z)\,dx.
\end{equation}
Let us consider a real valued function $g\in C_c^{\infty}(\mathbb{R})$ such that $0\leq g\leq 1$, $g\equiv 1$ in $[-1,1]$ and $g\equiv 0$ in $(-\infty,-2]\cup[2,\infty)$. Define the function $\phi_h:=g(\frac{z}{h})\phi$ and $\phi_{h,t}:=\text{min}\,\big\{z+t\phi_h,v\big\}$ with $h\geq 1$ and $t>0$ for a given non-negative $\phi\in C_c^{1}(\Omega)$. 
Then by the inequality (\ref{ineq2}), we have
\begin{equation}\label{ineq3}
\int_{\Omega}\mathcal{A}(x,\nabla z)\cdot\nabla(\phi_{h,t}-z)\,dx\geq\int_{\Omega}f(x)g_k(z)(\phi_{h,t}-z)\,dx.
\end{equation}
By (H5), we have 
\begin{align*}
I &=c\int_{\Omega}|\nabla(\phi_{h,t}-z)|^{\gamma}\big\{\mathcal{\overline{A}}(x,\nabla\phi_{h,t},\nabla z)\big\}^{1-\frac{\gamma}{p}}w(x)\,dx\\
&\leq \int_{\Omega}\big\{\mathcal{A}(x,\nabla \phi_{h,t})-\mathcal{A}(x,\nabla z)\big\}\cdot\nabla(\phi_{h,t}-z)\,dx\\
&=\int_{\Omega}\mathcal{A}(x,\nabla\phi_{h,t})\cdot\nabla(\phi_{h,t}-z)\,dx-\int_{\Omega}\mathcal{A}(x,\nabla z)\cdot\nabla(\phi_{h,t}-z)\,dx\\
&\leq\int_{\Omega}\mathcal{A}(x,\nabla\phi_{h,t})\cdot\nabla(\phi_{h,t}-z)\,dx-\int_{\Omega}f(x)g_k(z)\big(\phi_{h,t}-z\big)\,dx\,\big(\text{using} (\ref{ineq3})\big)
\end{align*}
Therefore, 
\begin{equation}\label{ineq4}
\begin{gathered}
I-\int_{\Omega}f(x)\big(g_k(\phi_{h,t})-g_k(z)\big)\big(\phi_{h,t}-z\big)\,dx\\
\leq \int_{\Omega}{\mathcal{A}(x,\nabla \phi_{h,t})}\cdot\nabla(\phi_{h,t}-z)\,dx-\int_{\Omega}f(x)g_k(\phi_{h,t})\big(\phi_{h,t}-z\big)\,dx\\
=\int_{\Omega}g(x)\,dx-\int_{\Omega}f(x)g_k(\phi_{h,t})\big(\phi_{h,t}-z-t\phi_h\big)\,dx+t\int_{\Omega}\mathcal{A}(x,\nabla\phi_{h,t})\cdot\nabla\phi_h\,dx-\\t\int_{\Omega}f(x)g_k(\phi_{h,t})\phi_h\,dx,
\end{gathered}
\end{equation}
where $$g(x):=\mathcal{A}(x,\nabla\phi_{h,t})\cdot\nabla(\phi_{h,t}-z-t\phi_h).$$
Let us denote by $$g_v(x):=\mathcal{A}(x,\nabla v)\cdot\nabla(\phi_{h,t}-z-t\phi_h).$$
Set $\Omega=S_v\cup S_v^{c},$ where $S_v := \{x\in\Omega:\phi_{h,t}(x)=v(x)\}$ and $S_v^{c} := \Omega\setminus S_v.$ Observe that $g(x)=g_v(x)=0$ on $S_v^{c}$ and $g(x)=g_v(x)$ on $S_v$. This gives from (\ref{ineq4}),
\begin{equation}\label{ineq5}
\begin{gathered}
I-\int_{\Omega}f(x)\big(g_k(\phi_{h,t})-g_k(z)\big)(\phi_{h,t}-z)\,dx\\
=\int_{\Omega}g_v(x)\,dx-\int_{\Omega}f(x)g_k(\phi_{h,t})\big(\phi_{h,t}-z-t\phi_h\big)\,dx+t\int_{\Omega}\mathcal{A}(x,\nabla\phi_{h,t})\cdot\nabla\phi_h\,dx-\\t\int_{\Omega}f(x)g_k(\phi_{h,t})\phi_h\,dx.
\end{gathered}
\end{equation}
Since $v$ is a super-solution of (\ref{main}), choosing $(z+t\phi_h-\phi_{h,t})$ as a test function in (\ref{supsolution}) and using the fact that $\phi_{h,t}=v$ on $S_v$, we obtain $$\int_{\Omega}g_v(x)\,dx-\int_{\Omega}f(x)g_k(\phi_{h,t})\big(\phi_{h,t}-z-t\phi_h\big)\,dx\leq 0.$$
Since $I\geq 0$ and $\phi_{h,t}-z\leq t\phi_h$, using the inequality (\ref{ineq5}), we get
\begin{align*}
\int_{\Omega}\mathcal{A}(x,\nabla \phi_{h,t})\cdot\nabla\phi_h\,dx-\int_{\Omega}f(x)g_k(\phi_{h,t})\phi_h\,dx\geq -\int_{\Omega}f|g_k(\phi_{h,t})-g_{k}(z)|\phi_h\,dx.
\end{align*}
Therefore letting $t\to 0$, we obtain
$$
\int_{\Omega}\mathcal{A}(x,\nabla z)\cdot\nabla\phi_h\,dx-\int_{\Omega}f(x)g_k(z)\phi_{h}\,dx\geq 0.
$$
As $h\to\infty$, we obtain
$$
\int_{\Omega}\mathcal{A}(x,\nabla z)\cdot\nabla\phi\,dx\geq \int_{\Omega}f(x)g_k(z)\phi\,dx.
$$
Hence the proof.
\end{proof} 
\begin{proof}[Proof of Theorem \ref{uniquenss delta greater than one}]
Suppose $u,v\in W^{1,p}_{loc}(\Omega,w)$ both are solutions of the problem (\ref{main}). Then, we can assume that $u$ is a sub-solution and $v$ is a super-solution of (\ref{main}). By the given condition on $f$, one can use Lemma \ref{unilemma1} to get the existence of $z\in\mathbb{K}$ satisfying the inequality (\ref{ineq1}). Let $\epsilon=2k^{-\frac{1}{\delta}}$ for $k>0$. Since $u=0$ on $\partial\Omega$, one can use Theorem \ref{bdythm} to obtain $(u-z-\epsilon)^{+}\in X$. Applying Lemma \ref{unilemma1}, for any $\eta>0$, by standard density arguments one has
\begin{equation}\label{Ineq1}
\int_{\Omega}\mathcal{A}(x,\nabla z)\cdot\nabla T_{\eta}\big((u-z-\epsilon\big)^{+})\,dx\geq \int_{\Omega}f(x)g_k(z)T_{\eta}\big((u-z-\epsilon)^{+}\big)\,dx.
\end{equation}
Since $(u-z-\epsilon)^{+}\in X$, there exists a sequence $\phi_n\in C_{c}^\infty(\Omega)$ such that $\phi_n\to (u-z-\epsilon)^{+}$ in $||\cdot||_X$. Denote by $$\psi_{n,\eta}:=T_{\eta}\big(\text{min}\,\big\{(u-z-\epsilon)^{+},\phi_{n}^{+}\big\}\big)\in X\cap L_c^{\infty}(\Omega),$$ and since $u$ is a sub-solution of (\ref{main}), we obtain 
$$
\int_{\Omega}\mathcal{A}(x,\nabla u)\cdot\nabla\psi_{n,\tau}\,dx\leq\int_{\Omega}\frac{f}{u^\delta}\psi_{n,\tau}\,dx.
$$
Since $w|\nabla u|^p$ is integrable in the support of $(u-z-\epsilon)^+$, one can pass to the limit as $n\to\infty$ and obtain
\begin{equation}\label{Ineq2}
\int_{\Omega}\mathcal{A}(x,\nabla u)\cdot\nabla T_{\eta}\big((u-z-\epsilon)^{+}\big)\,dx\leq\int_{\Omega}\frac{f}{u^\delta}T_{\eta}\big((u-z-\epsilon)^{+}\big)\,dx.
\end{equation}
By using (\ref{Ineq1}), (\ref{Ineq2}), the fact $\epsilon>k^{-\frac{1}{\delta}}$ together with (H5), we obtain for $\gamma := \text{max}\big\{p,2\big\}$,
\begin{align*}
&\int_{\Omega}|\nabla T_{\eta}\big((u-z-\epsilon)^{+}\big)|^{\gamma}\big(|\nabla u|^p+|\nabla z|^p\big)^{1-\frac{\gamma}{p}}w(x)\,dx\\
&\leq\int_{\Omega}\big\{\mathcal{A}(x,\nabla u)-\mathcal{A}(x,\nabla z)\big\}\cdot\nabla T_{\eta}\big((u-z-\epsilon)^{+}\big)\,dx\\
&\leq \int_{\Omega}f(x)\Big(\frac{1}{u^\delta}-g_k(z)\Big)T_{\eta}\big((u-z-\epsilon)^{+}\big)\,dx\\
&\leq\int_{\Omega}f(x)\big(g_k(u)-g_k(z)\big)T_{\eta}\big((u-z-\epsilon)^{+}\big)\,dx\leq 0.
\end{align*}
Since $\eta>0$ is arbitrary, we have $u\leq z+2k^{-\frac{1}{\delta}}\leq v+2k^{-\frac{1}{\delta}}.$ Letting $k\to\infty$, we get $u\leq v$ a.e. in $\Omega$. Arguing similarly we obtain $v\leq u$ a.e. in $\Omega$. Hence $u\equiv v$.
\end{proof}
\section*{Acknowledgement}
The author would like to show his sincere gratitude to Dr. Kaushik Balfor somefruitful discussion on the topic. The author was supported by NBHM FellowshipNo: 2-39(2)-2014 (NBHM-RD-II-8020-June 26, 2014).




\end{document}